\documentclass[11pt]{article}
\usepackage{graphicx}
\usepackage{imakeidx}

\usepackage[utf8x]{inputenc}
\usepackage{amssymb,pdfsync,epstopdf,verbatim,gensymb,amsmath,amsthm}
\usepackage{nccmath, mathtools,bm}
\usepackage{textcomp}
\usepackage{amsmath}
\usepackage{float}
\usepackage{subfig,boldline}
\usepackage{amsfonts}
\usepackage{epsfig,color, hyperref}
\usepackage{tabularx,longtable,adjustbox,siunitx,float,multirow,graphicx,booktabs}

\def\R {{\mathbb R}}

\def\cE {{\mathcal{E}}}

\def\H01{{H_0^1(\Omega)}}
\def\L2{{L^2(\Omega)}}

\usepackage{bmpsize}
\newtheorem{theorem}{Theorem}[section]

\newtheorem{proposition}{Proposition}
\newtheorem{algorithm}{Algorithm}[section]
\DeclareMathOperator*{\argmin}{arg\,min}

\newcommand{\TT}{\bar{T}}
\newcommand{\NN}{\bar{N}}
\newcommand{\LL}{\bar{L}}
\newcommand{\CC}{\bar{C}}
\newcommand{\bps}{\boldsymbol{\psi}}

\newcommand{\bX}{\boldsymbol{X}}

\newcommand{\bF}{\boldsymbol{F}}
\newcommand{\bW}{\boldsymbol{W}}
\newcommand{\bS}{\boldsymbol{\sigma}}

\newcommand{\bT}{\boldsymbol{\theta}}

\newcommand{\bU}{\boldsymbol{U}}

\newcommand{\Beq}{\begin{equation}}
\newcommand{\Eeq}{\end{equation}}
\newcommand{\beq}{\begin{equation*}}
\newcommand{\eeq}{\end{equation*}}
\newcommand{\bal}{\begin{align}}
\newcommand{\eal}{\end{align}}

\renewcommand{\L}{\langle}

\newcommand{\bp}{\begin{prob}}
\newcommand{\ep}{\end{prob}}
\newcommand{\bpr}{\begin{proof}}
\newcommand{\epr}{\end{proof}}

\DeclareRobustCommand{\ind}[1]{#1\index{#1}}

\newcommand{\bel}[1]{\begin{equation}\label{#1}}
\newcommand{\ee}{\end{equation}}

\numberwithin{equation}{section}
\numberwithin{proposition}{section}
\counterwithin{figure}{section}

\graphicspath{{./figures/}}

\makeindex

\begin{document}
\title{{\bf Optimal personalized therapies in colon-cancer induced immune response using a Fokker-Planck framework  }}
\author{
{\sc Souvik Roy}$^{1}$, {\sc Suvra Pal}$\,^2$,\\
\mbox{}\\
{\small $^1$  Department of Mathematics, The University of Texas at Arlington, Arlington, TX 76019, USA} \\ 
{\small $^2$ Department of Mathematics, The University of Texas at Arlington, Arlington, TX 76019, USA}\\}

\date{Email: souvik.roy@uta.edu $^1$; suvra.pal@uta.edu $^2$}

\maketitle

\begin{abstract}
In this paper, a new \ind{stochastic} framework to determine \ind{optimal} combination therapies in \ind{colon cancer}-induced  \ind{immune} response is presented. The dynamics of colon cancer is described through an It\"o \ind{stochastic} process, whose \ind{probability density function} evolution is governed by the \ind{Fokker-Planck} equation. An \ind{open-loop} \ind{control} \ind{optimization} problem is proposed to determine the optimal combination therapies. Numerical results with \ind{combination therapies} comprising of the \ind{chemotherapy} drug \ind{Doxorubicin} and \ind{immunotherapy} drug \ind{IL-2} validate the proposed framework.

\end{abstract}



\subsection{Introduction}

Colon cancer is a leading cause of global cancer related deaths \cite{Fitz2013}. The lack of early symptoms forces the detection of colon cancer to take place at the metastatic phase of the cancer \cite{Schm2012}. Thus, it becomes important to devise fast and accurate treatment strategies for cure. In this context, combination therapies have been clinically shown to be an effective strategy for combating cancer, in comparison to mono-therapy (see \cite{Mokhtari2017}).  Conventional mono-therapeutic techniques are indiscriminate in choosing actively growing cells that leads to the death of not only cancerous cells but also healthy cells. E.g., chemotherapy drugs can be toxic and lead to multiple side effects and risks, by weakening the patient's immune system targeting the bone-marrow cells \cite{Patridge2001}. This leads to increased susceptibility to secondary infections. However, with combination therapies, even though it might still be toxic if there is a presence of a chemotherapy drug, the toxicity effect is significantly diminished due to different targets being affected. Moreover, since the effect of combination therapies work in synergy, lower dosages of the individual drugs in the combination therapy are required for \ind{treatment}, which further reduces the toxic effects \cite{Mokhtari2013}.

To develop optimal combination therapies in colon cancer, it is important to understand relevant \ind{biomarkers} that govern the progression of the cancer \cite{Chang2021}. One such biomarker is the immune response to the cancerous cells.  It has been observed that onsite immune reactions at the tumor locations and prognosis are highly correlated and is independent of the size of the tumor \cite{Galon2006}. Furthermore, a high expression
of various \ind{immune pathways}, like Th1, Th17, are associated with a poor prognosis or prolonged disease-free survival for patients with colon cancer \cite{Tosolini2011}. Thus, it becomes important to develop optimal personalized treatments for colon cancer patients, taking into account the various types of \ind{immune cells}, their numbers and interaction with the cancer cells and each other. Since drugs for treatment of colon cancer, like chemotherapeutic drugs, immunotherapies, are quite costly, it is expensive to perform in-vitro and vivo experimental studies for testing effects of such drugs. An alternative cost effective option is to develop computational frameworks for testing optimal combination drug dosages \cite{Chang2022}. This can be done using \ind{pharmacokinetic} models, that are given by a set of ordinary or \ind{partial differential equations} (ODE or PDE) \index{ordinary differential equations}.

There are several dynamical models that have been developed to  describe the immune response and interactions in colon cancer and associated therapies. In \cite{Ballesta2012}, a combined compartmental model with that of irinotecan is used to describe the physiology of colon cancer. A set of ODEs were used to represent a 3D structure of colon cancer in \cite{Johnston2007}. In \cite{Koma2002}, the authors use  describe the phenomenon of \ind{tumorigenesis} in colon cancer using mathematical modeling. The authors in \cite{Lo2013}, describe the initiation of colon cancer and its link with colitis using a mathematical framework. In \cite{Powa2013}, the authors use a pharmacokinetic \ind{cellular automata} model to incorporate the \ind{cytotoxic effects} of chemotherapy drugs. In \cite{Haupt2021}, \ind{dynamical systems} are used to describe multiple pathways in colon cancer. For a detailed review of various mathematical models for colon cancer-induced immune response, we refer the reader to \cite{Ballesta2014}.

We contribute to the field of pharmacokinetic cancer research by presenting an effective approach to develop \ind{personalized therapies} for colon cancer-induced immune response. The starting point of this estimation process is the dynamical model for colon cancer-induced immune response, given in \cite{pillis2014}. The model describes the evolution of four variables: the \ind{tumor cell} count, concentration of the \ind{Natural Killer cells}, concentration of the CD$8^+$ cells, concentration of the remaining lymphocytes. We extend the dynamical model in \cite{pillis2014} to an \ind{It\^o stochastic} ODE system that takes into account the randomness of the colon cancer-induced immune response dynamics. We also model a combination therapy comprising of a chemotherapy drug and an immunotherapy drug as a control in the stochastic ODE system. The goal is to obtain an optimal combination therapy strategy that drives the cancer-induced immune response state. Since solving an optimal control problem using stochastic state equations is difficult, we use a more convenient framework for obtaining optimal combination therapies through the Fokker-Planck (FP) equations, that governs the evolution of the joint probability density function (PDF) associated to the random variables in stochastic process. Such FP control frameworks have been used for problems arising in control of collective \index{collective motion} and \ind{crowd motion} \cite{Roy2016,Roy2018a}, investigating \ind{pedestrian motion} from a game theoretic perspective \index{game theory} \cite{Roy2017}, reconstructing cell-membrane potentials \cite{Annun2021}, \ind{mean field control} problems \cite{Borzi2020}, controlling production of Subtilin \cite{Thal2016}. Very recently, in \cite{Roy_cancer1,Roy_cancer2}, FP frameworks were used for \ind{parameter estimation} in colon cancer-induced dynamical systems. Till date, this is the first work that considers the FP framework to devise \ind{optimal combination treatment strategies} for colon cancer-induced immune response.

In the next section, we describe a a FP control framework for devising optimal treatments in colon cancer-induced immune response systems. Section \ref{sec:theory} is concerned with the theoretical properties of the FP optimization problems. In Section \ref{sec:numerical_FP}, we describe the numerical discretization of the optimality system. In Section \ref{sec:results}, we obtain the optimal combination therapies, comprising of Doxorubicin and IL-2 for two types of simulated colon cancer patients, and show the correspondence of the results with experimental findings. We end with a section of conclusions.

\subsection{The control framework based on Fokker-Planck equations}\label{sec:FP}

The interaction between immune cells and colon cancer growth in a patient can be modeled using a coupled system of ODEs. The starting point is the set of equations given in the paper by dePillis \cite{pillis2014}. The following are the variables associated with different types of cell populations that appear as model variables
\begin{enumerate}

\item $T(\tau)$-  tumor cell population number (cells).
\item $N(\tau)$-  natural killer (NK) cells concentration per liter of blood (cells/L).
\item $L(\tau)$-  cytotoxic T \ind{lymphocytes} (CD8$^+$) concentration per liter of blood (cells/L).
\item $C(\tau)$-  other lymphocytes concentration per liter of blood (cells/L).

\end{enumerate}

The governing system of ODEs, representing the dynamics of the above defined cell populations are given as follows
\begin{equation}\label{eq:ODE}
\begin{aligned}
&\dfrac{dT}{d\tau} = aT-abT^2 - DT-cNT- \alpha_1 u_1(t)T,~ T(0) = T_0,\\
&\dfrac{dN}{d\tau} = eC-fN-pNT- \alpha_2 u_1(t)N + \beta_2u_2(t)N,~ N(0) = N_0,\\
&\dfrac{dL}{d\tau} = mL + j\dfrac{T}{k+T}L-qLT+(r_1N+r_2C)T - \alpha_3 u_1(t)L + \beta_3u_2(t)L,~ L(0) = L_0,\\
&\dfrac{dC}{d\tau} = \alpha-\beta C- \alpha_4 u_1(t)C ,~ C(0) = C_0,
\end{aligned}
\end{equation}
where $D = d\dfrac{1}{s(T/L)^l+1}$, $T_0,N_0,L_0,C_0$ represent the initial conditions for $T,N,L,C$, respectively, and $u_1,u_2$ represent dosages of Doxorubicin and IL-2, respectively. 

The parameters of the system \eqref{eq:ODE} are defined in \cite{pillis2014}.
To stabilize the solutions of the numerical methods, we non-dimensionalize the ODE system \eqref{eq:ODE} using the following non-dimensionalized state variables and parameters
\begin{equation}\label{eq:Nvar}
\begin{aligned}
& \bar{T} = k_1 T,~ \bar{N} = k_2 N,~ \bar{L} = k_3 L,~ \bar{C} = k_4 C,~t = k_5 \tau,\\
& \bar{a} = \dfrac{a}{k_5},~\bar{b} = \dfrac{b}{k_1},~\bar{c} = \dfrac{c}{k_2k_5},~\bar{d} = \dfrac{d}{k_5},~\bar{e} = \dfrac{ek_2}{k_4k_5},~\bar{f} = \dfrac{f}{k_5},~\bar{p} = \dfrac{10^{10}p}{k_1k_5},~\bar{m} = \dfrac{m}{k_5},~ \bar{j} = \dfrac{j}{k_5},~\bar{k} = k_1k,\\
&\bar{q} = \dfrac{10^8q}{k_5 k_1},~\bar{r}_1 = \dfrac{r_1k_3}{k_1k_2k_5},~ \bar{r}_2 = \dfrac{r_2k_3}{k_1k_4k_5},~ \bar{s} = 250 s,~ \bar{\alpha} = \dfrac{\alpha k_4}{k_5},~\bar{\beta} = \dfrac{\beta}{k_5} ~\bar{l} = l,\\
&\bar{\alpha}_i = \dfrac{\alpha_i }{k_5},~i = 1,\cdots,4,\\
&\bar{\beta}_i = \dfrac{\beta_i}{k_5}~i = 1,2.
\end{aligned}
\end{equation}
Then the transformed \ind{non-dimensional} ODE system is given as follows
\begin{equation}\label{eq:NODE}
\begin{aligned}
&\dfrac{d\TT}{dt} = \bar{a}\TT-\bar{a}\bar{b}\TT^2 -\bar{D}\TT-  \bar{c}\NN\TT-\bar{\alpha}_1 \bar{u}_1(t)\TT,~ \TT(0) = \TT_0\\
&\dfrac{d\NN}{dt} = \bar{e}\CC-\bar{f}\NN-10^{-10}\bar{p}\NN\TT - \bar{\alpha}_2 \bar{u}_1(t)\NN + \bar{\beta}_2 \bar{u}_2(t)\NN,~ \NN(0) = \NN_0\\
&\dfrac{d\LL}{dt} = \bar{m}\LL + \bar{j}\dfrac{\TT}{\bar{k}+\TT}\LL-10^{-8}\bar{q}\LL\TT+(\bar{r}_1\NN+\bar{r}_2\CC)\TT- \bar{\alpha}_3 \bar{u}_1(t)\LL + \bar{\beta}_3 \bar{u}_2(t)\LL,~ \LL(0) = \LL_0\\
&\dfrac{d\CC}{dt} = \bar{\alpha}-\bar{\beta} \CC- \bar{\alpha}_4 \bar{u}_1(t)\CC ,~ \CC(0) = \CC_0,
\end{aligned}
\end{equation}
where $\bar{D} = \bar{d}\dfrac{(\LL/\TT)^{\bar{l}}}{4\bar{s}\cdot10^{-3}\cdot (k_1/k_3)^{\bar{l}}+(\LL/\TT)^{\bar{l}}}$.

The compact form of the aforementioned system of ODEs, given in \eqref{eq:NODE}, is as follows:
\begin{equation}\label{eq:ODE_compact}
\begin{aligned}
&\dfrac{d\bX}{dt} = \bF(\bX,\bU),\\
&\bX(0) = \bX_0,
\end{aligned}
\end{equation}
where $\bX(t) = (\TT(t),\NN(t),\LL(t),\CC(t))^T$.

We extend the ODE system \eqref{eq:NODE} to include stochasticity present in the dynamics. For this purpose, we consider the  It\^o stochastic differential equation corresponding to \eqref{eq:NODE} 
\begin{equation}\label{eq:ItoODE}
\begin{aligned}
&d\TT = (\bar{a}\TT(1-\bar{b}\TT) - \bar{c}\NN\TT-\bar{D}\TT- \bar{\alpha}_1 \bar{u}_1(t)\TT)~dt + \sigma_1(\TT) ~dW_1(t),~ \TT(0) = \TT_0,\\
&d\NN = (\bar{e}\CC-\bar{f}\NN-10^{-10}\bar{p}\NN\TT- \bar{\alpha}_2 \bar{u}_1(t)\NN + \bar{\beta}_2 \bar{u}_2(t)\NN)~dt + \sigma_2 (\NN)~ dW_2(t),~ \NN(0) = \NN_0,\\
&d\LL = (\bar{m}\LL + \bar{j}\dfrac{\TT}{\bar{k}+\TT}\LL-10^{-8}\bar{q}\LL\TT+(\bar{r}_1\NN+\bar{r}_2\CC)\TT- \bar{\alpha}_3 \bar{u}_1(t)\LL + \bar{\beta}_3 \bar{u}_2(t)\LL)~dt + \sigma_3(\LL) ~dW_3(t),\\
&\LL(0) = \LL_0,\\
&d\CC= (\bar{\alpha}-\bar{\beta} \CC- \bar{\alpha}_4 \bar{u}_1(t)\CC )~dt + \sigma_4 (\CC)~ dW_4(t),~ \CC(0) = \CC_0,
\end{aligned}
\end{equation}
where $dW_i,~ i = 1,2,3,4$ are one-dimensional \ind{Wiener processes} and $\sigma_i,~ i=1,2,3,4$  are positive constants. The equation \eqref{eq:ItoODE} can be written using a compact notation as follows
 \begin{equation}\label{eq:ItoODE_compact}
\begin{aligned}
&d\bX= \bF(\bX,\bU)~dt + \bS(\bX)~d\bW(t) ,\\
&\bX(0) = \bX_0,
\end{aligned}
\end{equation}
where 
\[
d\bW(t) = 
\begin{pmatrix}
dW_1(t)&dW_2(t)&dW_3(t)&dW_4(t) \\
\end{pmatrix}^T
\]
is a four-dimensional Wiener process vector with \ind{stochastically independent} components and 
\[
\bS = diag
\begin{pmatrix}
\sigma_1 &\sigma_2 &\sigma_3 &\sigma_4 
\end{pmatrix}
\]
is the dispersion matrix.

We now describe the PDF of the stochastic process \eqref{eq:ItoODE_compact}, confined in a \ind{Lipschitz domain} $\Omega$, by virtue of a \ind{reflecting barrier} on $\partial \Omega$. This is motivated by the \ind{maximum cell carrying capacity}. Thus, $X(t) \in \Omega\subset\R^4_+ = \lbrace x\in\mathbb{R}^4:x_i \geq 0,~ i= 1,2,3,4 \rbrace$, 
 Let $x = (x_1,x_2,x_3,x_4)^T$. Define $f(x,t)$ as the PDF for the stochastic process described by \eqref{eq:ItoODE_compact}, i.e., $f(x,t)$ is the probability of $\bX(t)$ assuming the value $x$ at time $t$. Then the PDF of $\bX(t)$ evolves through the following Fokker-Planck (FP) equations
\begin{equation}\label{eq:FP}
\begin{aligned}
&\partial_t f(x,t)+\nabla \cdot (\bF(x,\bU)~f(x,t)) = 
\frac{1}{2}\nabla \cdot (\bS^2(x) \nabla f(x,t)),\\
&f(x,0) =  f_0(x),
\end{aligned}
\end{equation}
where $f_0(x)$ is non-negative with mass equals one,
 and
 $\bU(t) = (\bar{u}_1(t),\bar{u}_2(t))$ in the admissible set
\[
U_{ad} = \lbrace \bU \in L^2([0,T]): 0 \leq \bar{u}_i(t) \leq D_i,~ D_i>0,~ \forall t\in[0,T],~ i = 1,2\rbrace
\]
Here, the FP domain is $Q=\Omega\times(0,T_f)$, where $T_f$ is the final time and $f_0(x)$ represents the distribution of the initial state $X_0$ of the process. The FP equation \eqref{eq:FP} is associated with the \ind{no-flux boundary conditions}. To describe this, we write \eqref{eq:FP} in flux form as 
\begin{equation}\label{eq:FPflux}
\partial_t f(x,t)-\nabla\cdot \mathcal{H}=0, \qquad f(x,0) =  f_0(x) ,
\end{equation}
where the components of the flux $\mathcal{H}$ is given as follows
\begin{equation}\label{eq:flux_def}
\mathcal{H}_j(x,t;f) = \frac{\sigma_j^2}{2}\partial_{x_j} f-\bF_j(x,\bU)f, ~ j = 1,2,3,4.
\end{equation}
Then the no-flux boundary conditions are
\begin{equation}\label{eq:nfbc}
\mathcal{H}\cdot \hat n = 0 \qquad \mbox{ on } \partial\Omega\times(0,T_f),
\end{equation}
with $\hat n$ as the unit outward normal on $\partial\Omega$. 

To obtain the \ind{optimal combination therapy} function vector $\bU$, we solve the following optimization problem
\begin{equation}\label{eq:min_problem}
\begin{aligned}
\bU^* = \argmin_{\bU \in U_{ad}} J(f,\bU) := &\dfrac{\alpha}{2} \int_Q (f(x,t) - f^{*}(x,t))^2~dx + \dfrac{\nu_1}{2}\int_0^T\|u_1(t)\|^2~dt + \dfrac{\nu_2}{2}\int_0^T\|u_2(t)\|^2~dt,
\end{aligned}
\end{equation}
subject to the FP system \eqref{eq:FP},\eqref{eq:nfbc}, where the desired PDF is $f^*(x,t)$.

\subsection{Theoretical results} \label{sec:theory}
In this section, we describe some theoretical results related to the \ind{minimization problem} \eqref{eq:min_problem}. One can also find similar results in  \cite{MA,Roy2016,Roy2018a}. For this purpose, we denote the FP system  \eqref{eq:FP},\eqref{eq:nfbc} as $\cE(f_0,\bU)=0$. The \ind{existence} and \ind{uniqueness} of solutions of \eqref{eq:FP} is given in the following proposition
\begin{proposition}\label{th:proposition1}
Assume $f_0 \in H^1(\Omega)$ with $f_0$ non-negative, and $\bU\in U_{ad}$. Then, there exists an unique non-negative solution of $\cE(f_0,\bU)=0$ given by
$f \in L^2([0,T_f];H^1(\Omega)) \cap C([0,T_f];L^2(\Omega))$.
\end{proposition}

Under the assumptions of higher regularity of $\partial \Omega$, the boundary of $\Omega$, one can also obtain $H^2(\Omega)$ regularity of the solution of \eqref{eq:FP} (see \cite{Tao}). Next we state the \ind{conservativeness} property of \eqref{eq:FP}, which can be proved using straightforward applications of weak formulation, integration by parts, and divergence theorem for the flux $\mathcal{H}$.
\begin{proposition}\label{th:FPcons}
The FP system given in \eqref{eq:FP},\eqref{eq:nfbc} is conservative.
\end{proposition}
The next proposition states and proves the $L^2$ stability property of \eqref{eq:FP}.
\begin{proposition}\label{th:stability}
The FP system \eqref{eq:FP},\eqref{eq:nfbc} solution, given by $f$, satisfies the following $L^2$ \ind{stability} property
\begin{equation}\label{eq:FPstability}
\| f(t) \|_{L^2(\Omega)}
\le  \|  f_0 \|_{L^2(\Omega)}  \exp \left(  \|\bS^{-1}\|^2_2 N^2 t \right),
\end{equation}
where $N = \sup_{\Omega\times U} |F(x,\bT)|$.
\end{proposition}
\begin{proof}
Multiplying \eqref{eq:FP} with the test function $\psi=f(\cdot,t)$ and an integration by parts gives
\begin{equation}\label{eq:FPstability0}
\frac{\partial }{\partial t} \| f(t) \|^2_{L^2(\Omega)}
= -  \| \bS \nabla f(t) \|^2_{L^2(\Omega)}  + 2 \int\limits_\Omega (\bF f(t)) \cdot \bS^{-1}\bS \nabla f(t) \, dx.
\end{equation}

The last term in \eqref{eq:FPstability0} can be estimated using the \ind{Young's inequality},
$2bd \le kb^2 + d^2/k$ with $k = \|\bS^{-1}\|_2$, which is the $L^2$ matrix norm of $\bS^{-1}$. We then obtain the following
$$
\frac{\partial }{\partial t} \| f(t) \|^2_{L^2(\Omega)}
\le   \|\bS^{-1}\|^2_2 N^2 \|  f(t) \|^2_{L^2(\Omega)}  .
$$
Applying the \ind{Gronwall's inequality} gives the desired result.
\end{proof}

The above results imply that the map $\Lambda : U_{ad} \to C([0,T_f];H^1(\Omega))$ given by $f=\Lambda(\bU)$, is \ind{continuous} and \ind{Fr\'echet differentiable}. The main theoretical result in this work, which shows that there exists an optimal open-loop control $\bU^*$, is given below:
\begin{theorem}\label{th:existence}
Let $f_0 \in H^1(\Omega)$ and let $J$ be given as in \eqref{eq:min_problem}. Then there exists $(f^*,\bU^*) \in C([0,T_f];$ $H^1(\Omega)) \times U_{ad}$ with $f^*$ being a solution to $\cE(f_0,\bU^*)=0$ and $\bU^*$ minimizing $J$ in $U_{ad}$.
\end{theorem}
\begin{proof}
Since $J$ is bounded, a \ind{minimizing sequence} $(\bU^m)$ exists in $U_{ad}$. Moreover, $J$ is being \ind{coercive} and sequentially \ind{weakly lower semi-continuous} in $U_{ad}$, implies the \ind{boundedness} of this sequence. Due to the fact that $U_{ad}$ is a closed and convex subset of a Hilbert space, the \ind{sequence} $(\bU^m)$ contains a \ind{convergent}
\ind{subsequence} $(\bU^{m_l})$ in $U_{ad}$, such that $\bU^{m_l} \rightarrow \bU^*$.
Correspondingly, the sequences $(f^{m_l}), (\partial_t f^{m_l})$, where $f^{m_l}=\Lambda(\bT^{m_l})$,
are bounded in $L^2([0,T_f]; H^1(\Omega) ), ~L^2([0,T_f]; H^{-1}(\Omega))$, respectively. This implies the \ind{weak convergence} of the sequences to $f^*$ and $\partial_t f^* $, respectively. We next use the compactness result of \ind{Aubin-Lions} \cite{Lions1969} to obtain strongly convergent subsequence $(f^{m_k})$ in $L^2([0,T_f],L^2(\Omega))$.
Thus, the sequence $(\bF(\bU^{m_k})f^{m_k})$ in $L^2([0,T_f],L^2(\Omega))$ is weakly convergent. This implies that $f^*=\Lambda(\bU^*)$,
and $(f^*,\bU^*)$ is a minimizer of $J$.
\end{proof}

For the minimization problem \eqref{eq:min_problem}, the \ind{optimality system} can now be written as

\begin{equation}\label{eq:opt_FP}
\begin{aligned}
&\partial_t f(x,t)+\nabla \cdot (\bF(x,\bU)~f(x,t)) =
\frac{1}{2}\nabla \cdot (\bS^2 \nabla f(x,t)), \qquad \mbox{ in } \Omega\times(0,T_f),\\
&f(x,0) =  f_0(x), \qquad \mbox{ in } \Omega,\\
&\mathcal{H}\cdot \hat n = 0, \qquad \mbox{ on } \partial\Omega\times(0,T_f).
\end{aligned}\tag{FOR}
\end{equation}

\begin{equation}\label{eq:opt_adj}
\begin{aligned}
-\partial_t p(x,t)-f(x,t)(&\bF(x,\bT)\cdot \nabla p(x,t)) -
\frac{1}{2}\nabla \cdot (\bS^2 \nabla p(x,t)) = -\alpha (f(x,t)-f_i^*(x,t)), ~\mbox{ in } \Omega\times(0,T_f),\\
& p(x,T_f)=0, \qquad \mbox{ in } \Omega,\\
&\frac{\partial{p}}{\partial{n}}  = 0, \qquad  \mbox{ on } \partial\Omega\times(0,T_f).\\
\end{aligned}\tag{ADJ}
\end{equation}

\begin{equation}\label{eq:opt_cond}
\Big\langle{ \beta \bU - \int_\Omega\nabla_{\bU} \bF\cdot \nabla p , \bps-\bU }\Big\rangle_{L^2([0,T])} \geq 0,\qquad \forall \bps \in U_{ad}. \tag{OPT}
\end{equation}

The optimality system comprises of three sets of equations: the forward or \ind{state equation} that governs the FP dynamics \eqref{eq:opt_FP}, the \ind{adjoint equation} \eqref{eq:opt_adj}, and the \ind{optimality condition} \eqref{eq:opt_cond}. In the next section, we describe \ind{numerical schemes} to implement the optimality system.

\subsection{Numerical schemes}\label{sec:numerical_FP}

We consider the mesh $\lbrace\Omega_h\rbrace_{h>0}$ given by
$$
\Omega_h = \lbrace(x_1,x_2,x_3,x_4)\in\mathbb{R}^4:(x_{1i},x_{2j},x_{3k},x_{4l}) = (x_{10}+ih,x_{20}+jh,x_{30}+kh,x_{40}+lh)\rbrace,
$$
where $(i,j,k,l)\in
\lbrace{0,\hdots,N_{x_1}}\rbrace\times \lbrace{0,\hdots,N_{x_2}}\rbrace\times \lbrace{0,\hdots,N_{x_3}}\rbrace\times \lbrace{0,\hdots,N_{x_4}}\rbrace\cap\Omega$, and $N_{x_i}$ is the number of discretization points along the $i^{th}$ coordinate direction. Define $\delta{t}=T_f/N_t$ to be the temporal discretization step, where $N_t$ denotes the maximum number of temporal steps. This gives us the discretized domain for $\Omega$ as follows
$$
Q_{h,\delta{t}} = \lbrace{(x_{1i},x_{2j},x_{3k},x_{4l},t_m):\, (x_{1i},x_{2j},x_{3k},x_{4l})\in\Omega_h,~t_m=m\delta{t},~0\leq m\leq N_t}\rbrace.
$$
The value of $f(x,t)$ on $Q_{h,\delta{t}}$ is denoted as $f_{i,j}^m$. To solve the forward Fokker-Planck equation \eqref{eq:FP}, we use the scheme described in \cite{Roy_cancer2}, that is comprised of the spatial discretization using the \ind{Chang-Cooper} (CC) method \cite{CC}. The temporal derivative discretization is done using the four-step, alternate direction implicit \ind{Douglas-Gunn} (DG4) method. The combined scheme is referred to as the DG4-CC scheme. It has been shown in \cite{Roy_cancer2} that this scheme is \ind{positive}, conservative, stable, and \ind{second order convergent} in the $L^1$ norm.  Furthermore, we use the temporal D-G scheme for the time discretization in the first term, one sided \ind{finite difference} discretization for the second term, and \ind{central difference} for the third term on the left hand side of the adjoint equation \eqref{eq:opt_adj}.

To solve the optimization problem \eqref{eq:min_problem}, we use a \ind{projected} \ind{non-linear conjugate gradient} scheme (PNCG), as described in \cite{Roy_cancer1,Roy_cancer2}. Such a scheme has also been used for solving optimization problems related to controlling stochastic crowd motion \cite{Roy2016,Roy2018a}, studying \ind{avoidance} behavior of pedestrians using game theory \cite{Roy2017}, and \ind{cure rate models} \cite{Pal2020,Pal2021}. The PNCG scheme is described below:

\begin{algorithm}[PNCG scheme]\label{algo3}\ 
\begin{enumerate}
\item Input: Initial guess $u_0$. Compute $d_0 = -\nabla\hat{J}(u_0)_{H^1}$. Set $k=0$ and maximum number of iterations $k_{max}$, tolerance $TOL$.
\item If $(k<k_{max})$ do
\item Compute $u_{k+1} = P_{U}\left[ u_k + \alpha_k \, d_k \right]$, with $\alpha_k$ obtained by the Armijo line-search method.
\item Evaluate $g_{k+1} = \nabla\hat{J}(u_{k+1})_{H^1}$.
\item Evaluate $\beta_k^{HG} =\frac{1}{d_k^Ty_k}\left({y_k-2d_k\frac{\|y_k\|^2}{d_k^Ty_k}}\right)^Tg_{k+1}.$
\item Compute $d_{k+1}=-g_{k+1}+\beta_k^{HG}d_k$.
\item Check for $\|u_{k+1}-u_k\|_2 < TOL$. If yes, terminate.
\item Update $k=k+1$.
\item End if.
\end{enumerate}
\end{algorithm}

\subsection{Results}\label{sec:results}

This section describes the numerical results for obtaining optimal treatment strategies with the aforementioned FP framework. We choose our domain $\Omega = (0,6)^4$ and discretize it using $N_{x_i} = 51$ points for $i = 1,2,3,4$. The final time $T_f$ is chosen to be 10 and the maximum number of time steps $N_t$ is chosen to be 200. The values of the constants used in converting the ODE system \eqref{eq:ODE} to its non-dimensional form given in \eqref{eq:NODE} are given as $k_1 = 10^{-10},~k_2 = 10^{-5},~k_3 = 10^{-7},~k_4 = 10^{-8}$, and $k_5=1$. To obtain the target PDF $f^*(x,t)$, we simulate the ODE system \eqref{eq:NODE} with $\bar{u}_1,\bar{u}_2$ set to 0 and with the value of the non-dimensional parameters $(\bar{d},\bar{l},\bar{s}) = (2.1,1.1,1.25) $ that represents a patient with strong immune response \cite{pillis2014}. The values of the other parameters are taken from \cite{pillis2014}. After we obtain the trajectories of $\bar{T},\bar{N},\bar{L},\bar{C}$, we then choose 20 time points $t_i$, and at each point $(\bar{T}(t_i),\bar{N}(t_i),\bar{L}(t_i),\bar{C}(t_i))$ assign a Gaussian PDF given by a normal distribution with variance 0.05. We finally perform a 5D interpolation to obtain the desired PDF $f^*(x,t)$. Based on a statistical analysis of the \ind{dataset} given in \cite{pillis2014}, we choose \[
\sigma_i(x) = 0.5(x_i^{1.2}+0.001),~ i = 1,2,3,4.
\]
The regularization parameters are chosen to be $\alpha = 1,~ \nu_1 = \nu_2 = 0.01$.
The maximum tolerable dosage for Doxorubicin is taken to be $7$ mg/day and for IL-2 is taken to be $7.2\times 10^5$ IU IL-2/l/day, which implies, $D_1 = 7$ and $D_2 = 0.072$.

In Test Case 1, we simulate a tumor patient with a value of $(T(0),N(0),L(0),C(0)) = (2\times 10^9, 10^5, 10^7, 10^8)$.  This implies $(\bar{T}(0),\bar{N}(0), \bar{L}(0),\bar{C}(0)) = (0.2,1,1,1))$. We also choose the values of the non-dimensional parameters $(\bar{d},\bar{l},\bar{s}) = (1.3,2,10) $ that represents patient with \ind{weak immune system}.
The goal is to determine optimal drug dosages $\bar{u}_i,~ i = 1,2$ such that the tumor profile is given by the target PDF $f^*$.

\begin{figure}[H]
\centering
\subfloat[$T$ (without treatment)]{\includegraphics[width=0.25\textwidth, height=0.25\textwidth]{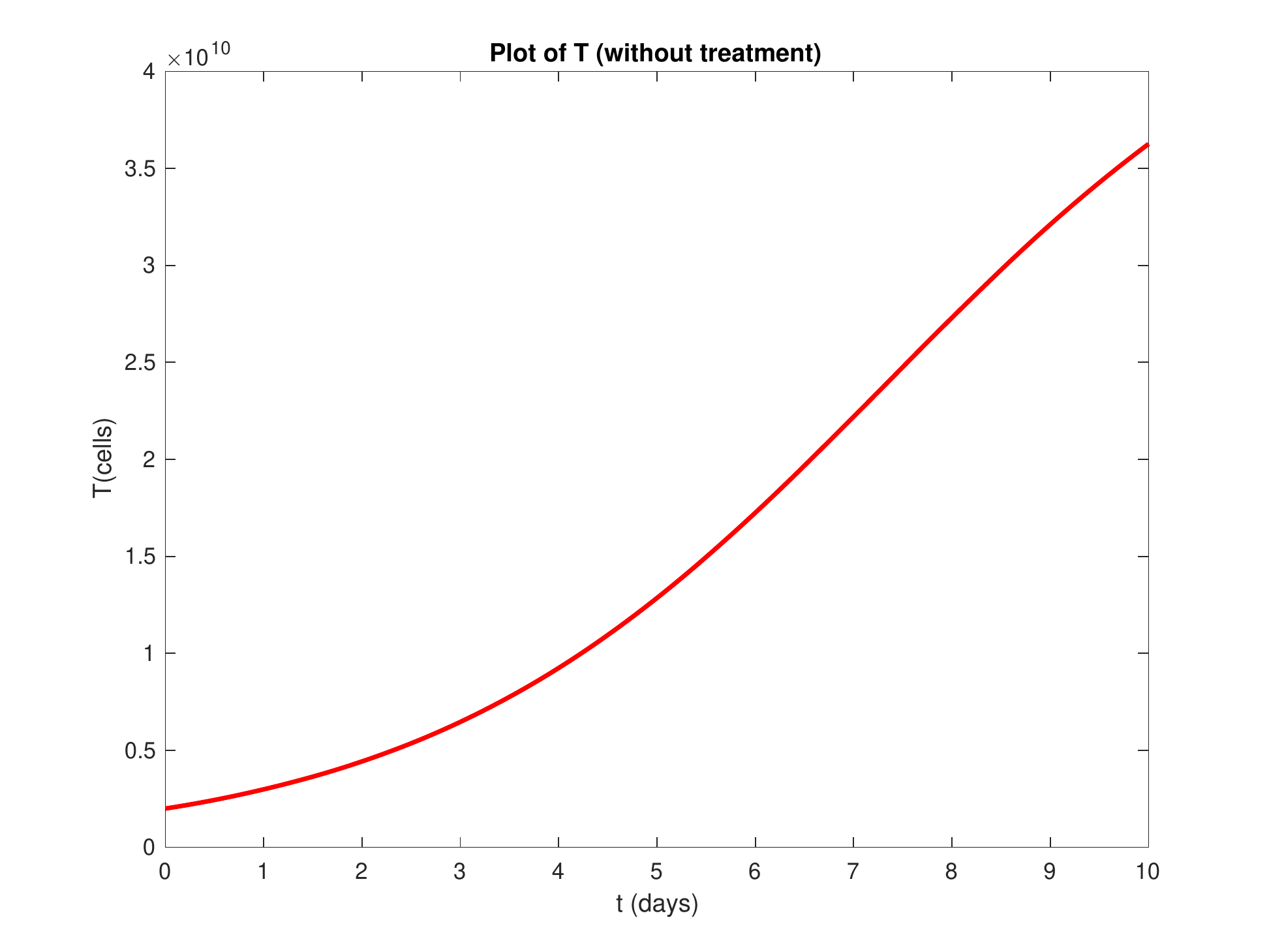}}
\subfloat[$T$ (with treatment)]{\includegraphics[width=0.25\textwidth, height=0.25\textwidth]{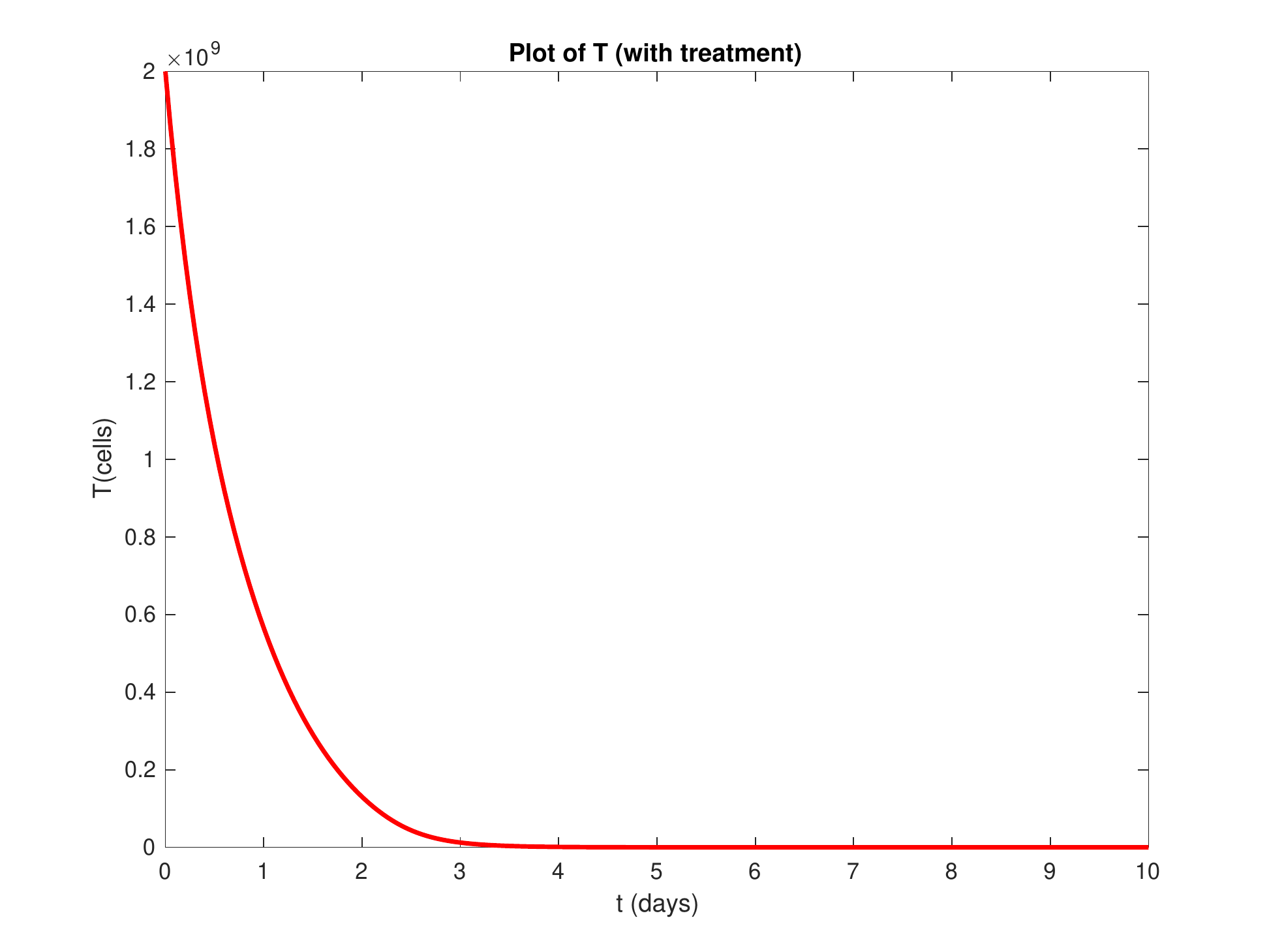}}
\subfloat[Doxorubicin profile]{\includegraphics[width=0.25\textwidth, height=0.25\textwidth]{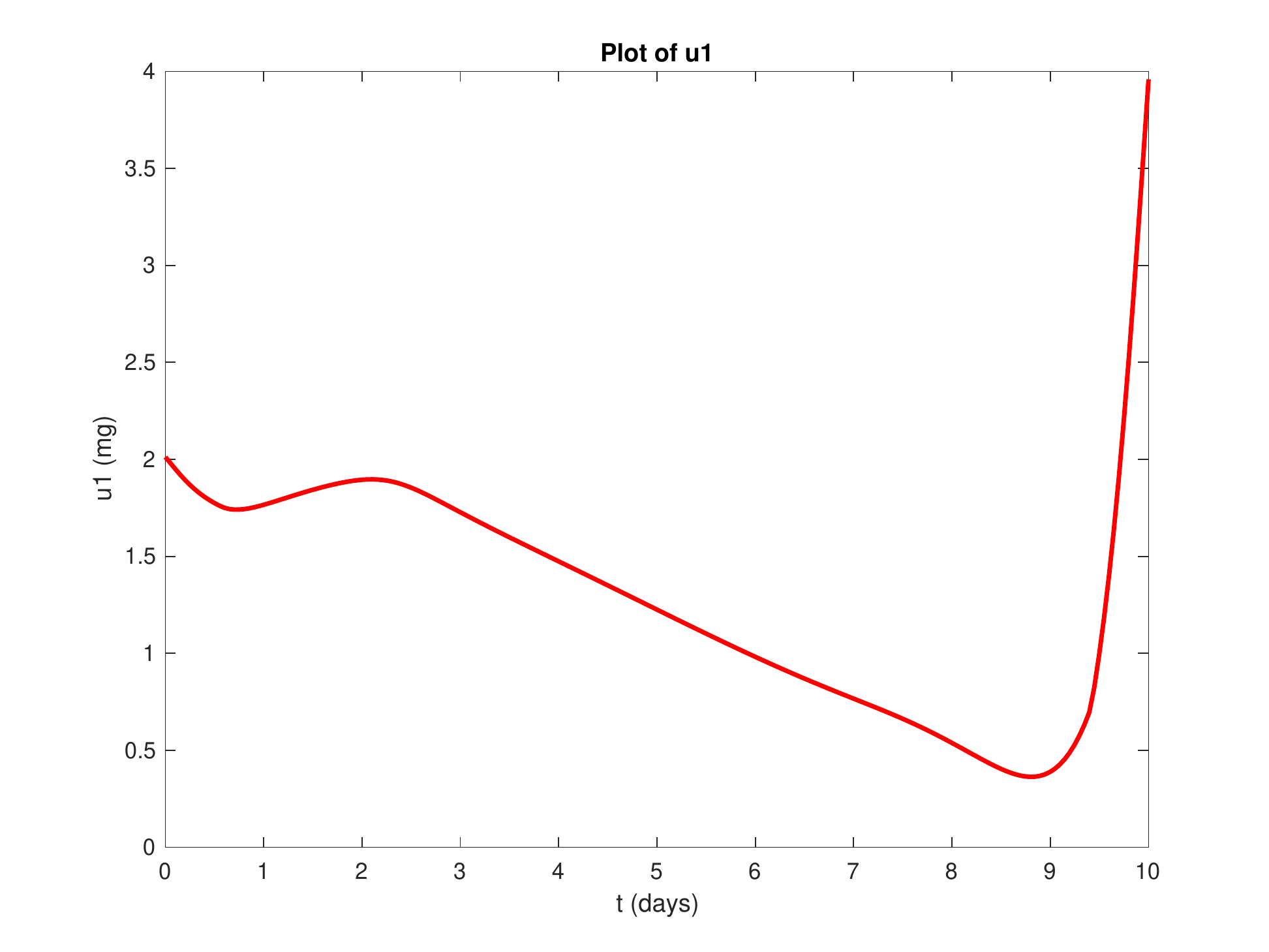}}
\subfloat[IL-2 profile]{\includegraphics[width=0.25\textwidth, height=0.25\textwidth]{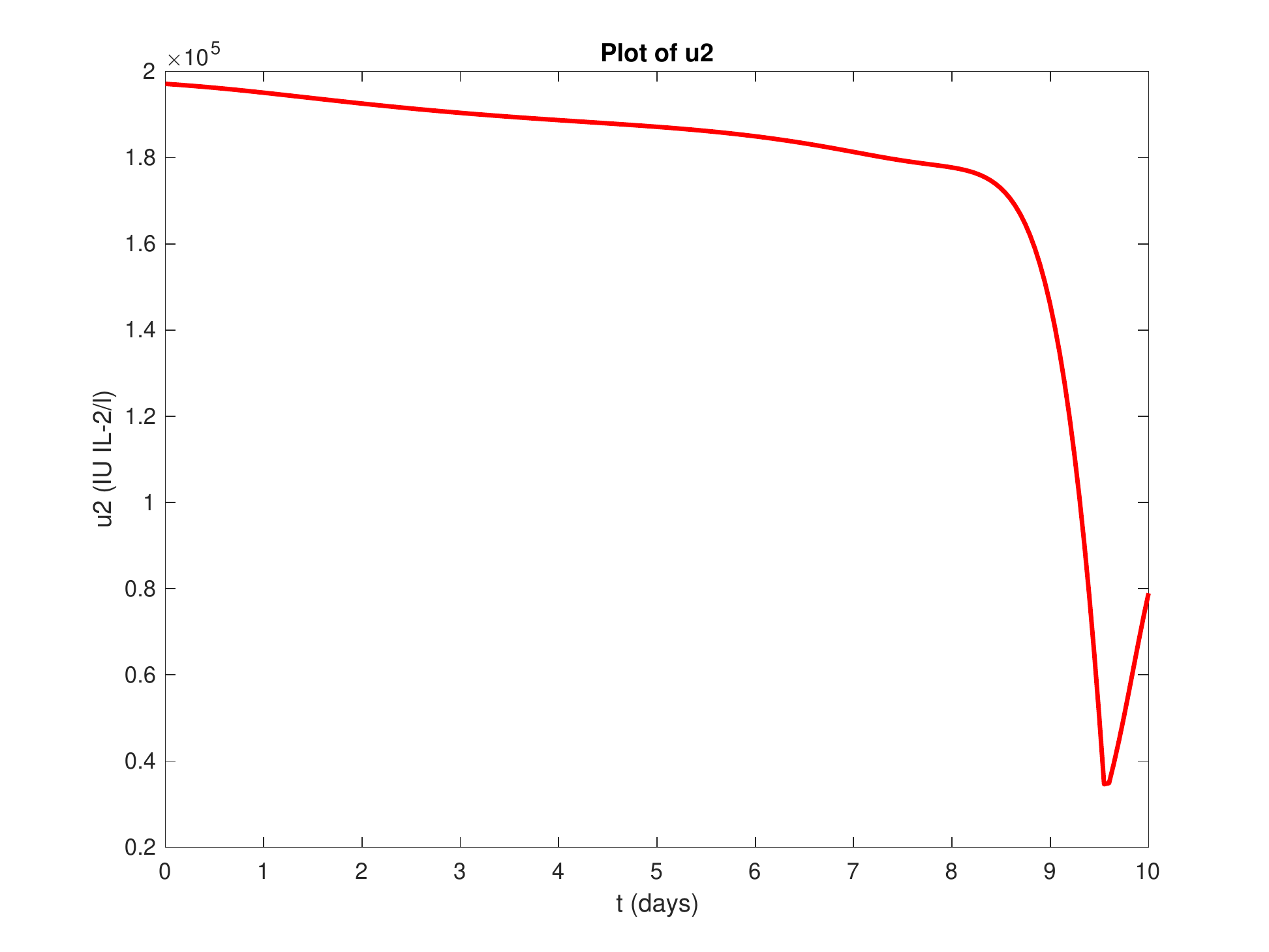}}\\

\caption{Test Case 1: Plots of the mean tumor profiles and drug dosages over the 10 day period}
    \label{fig:test_case1}
  \end{figure}
Figure \ref{fig:test_case1} presents the simulation results obtained with our framework. We observe that without the combination therapy, the mean tumor cell count will keep on increasing till it reaches the cell carrying capacity, ultimately leading to patient death. With the combination therapy, the mean tumor \ind{cell count} is brought back to diminishing levels. We also note the optimal dosage patterns for both Doxorubicin and IL-2 over time. Traditionally, Doxorubicin is administered every 21 days with a dosage of 142.5 mg \cite{pillis2009}, that translates to a total dosage of 70 mg for 10 days. However, from Figure \ref{fig:test_case1}, we note that the total dosage of Doxorubicin over 10 days is far less than 70 mg. Moreover, the observed daily dosage of IL-2 from Figure \ref{fig:test_case1} is less than the standard dosage of IL-2 over a day, which is $2.1\times10^6$ IU IL-2/l. This suggests that optimal combination therapies administered daily leads to lower total dosages and, thus, lower toxicity effects.

 In Test Case 2, we choose the same initial values of $\bar{T},\bar{N}, \bar{L},\bar{C}$, along with the values of the non-dimensional parameters $(\bar{d},\bar{l},\bar{s}) = (1.6,1.4,2) $ that represents patient with \ind{moderately strong immune system}.

\begin{figure}[H]
\centering
\subfloat[$T$ (without treatment)]{\includegraphics[width=0.25\textwidth, height=0.25\textwidth]{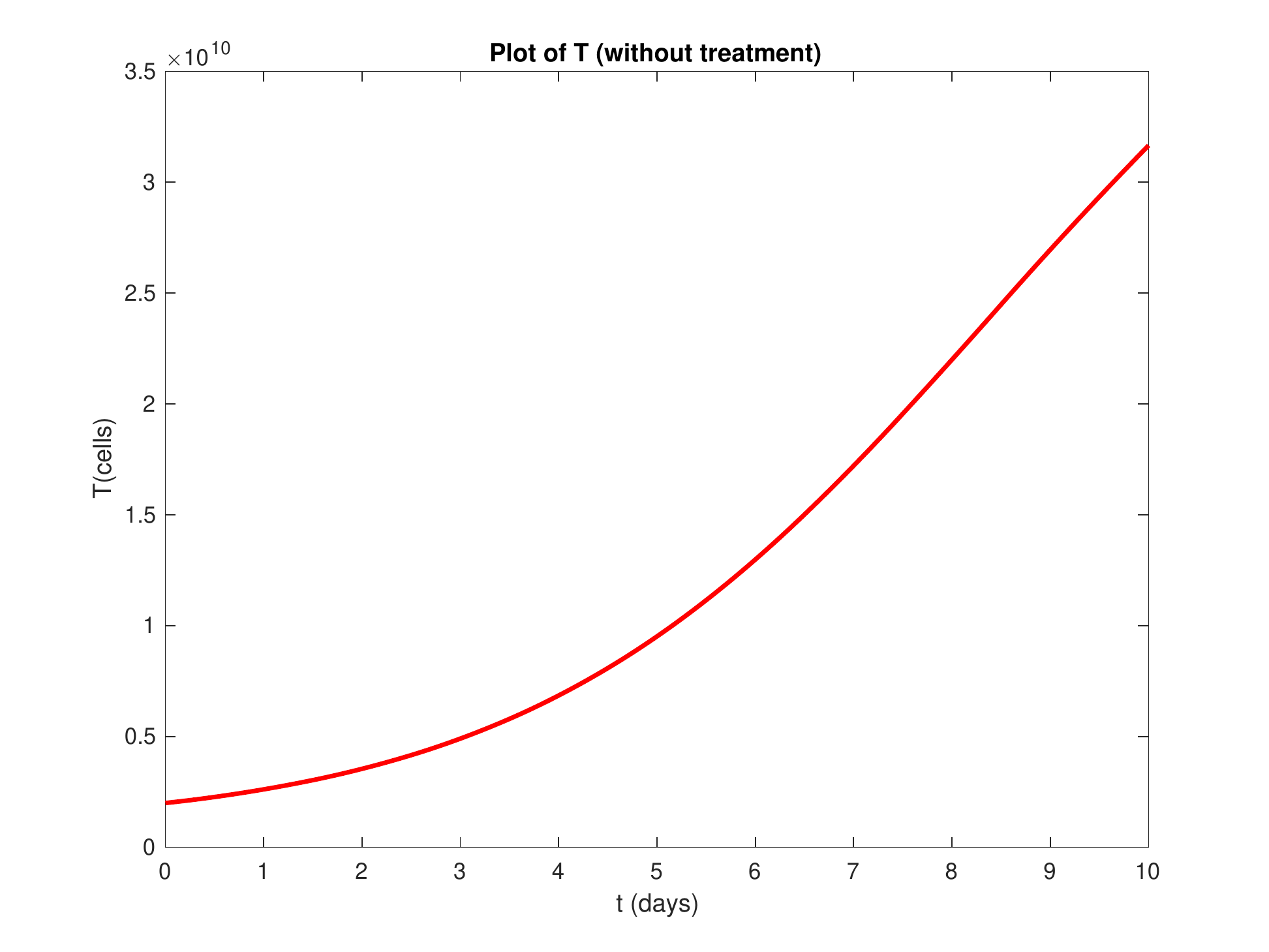}}
\subfloat[$T$ (with treatment)]{\includegraphics[width=0.25\textwidth, height=0.25\textwidth]{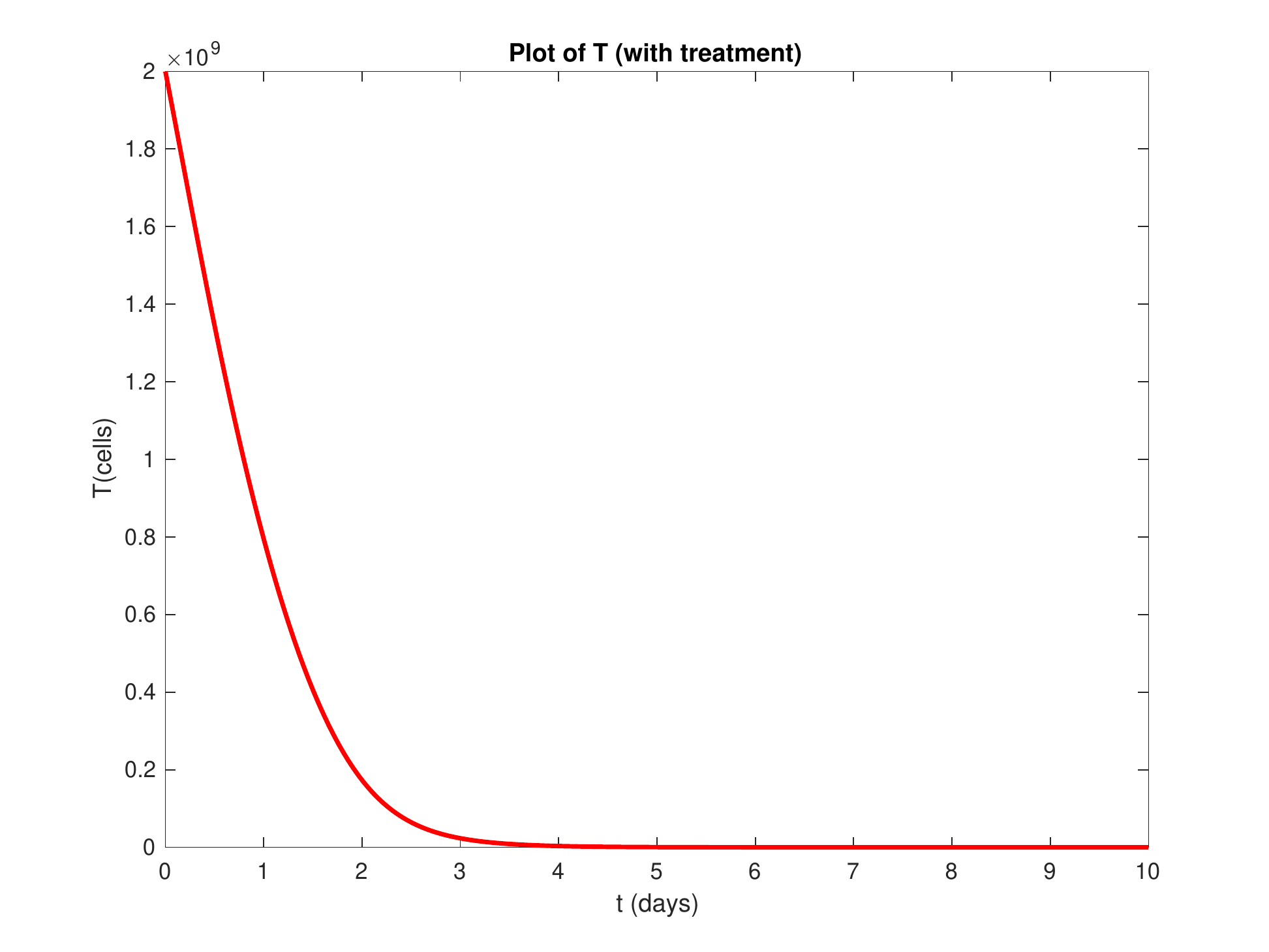}}
\subfloat[Doxorubicin]{\includegraphics[width=0.25\textwidth, height=0.25\textwidth]{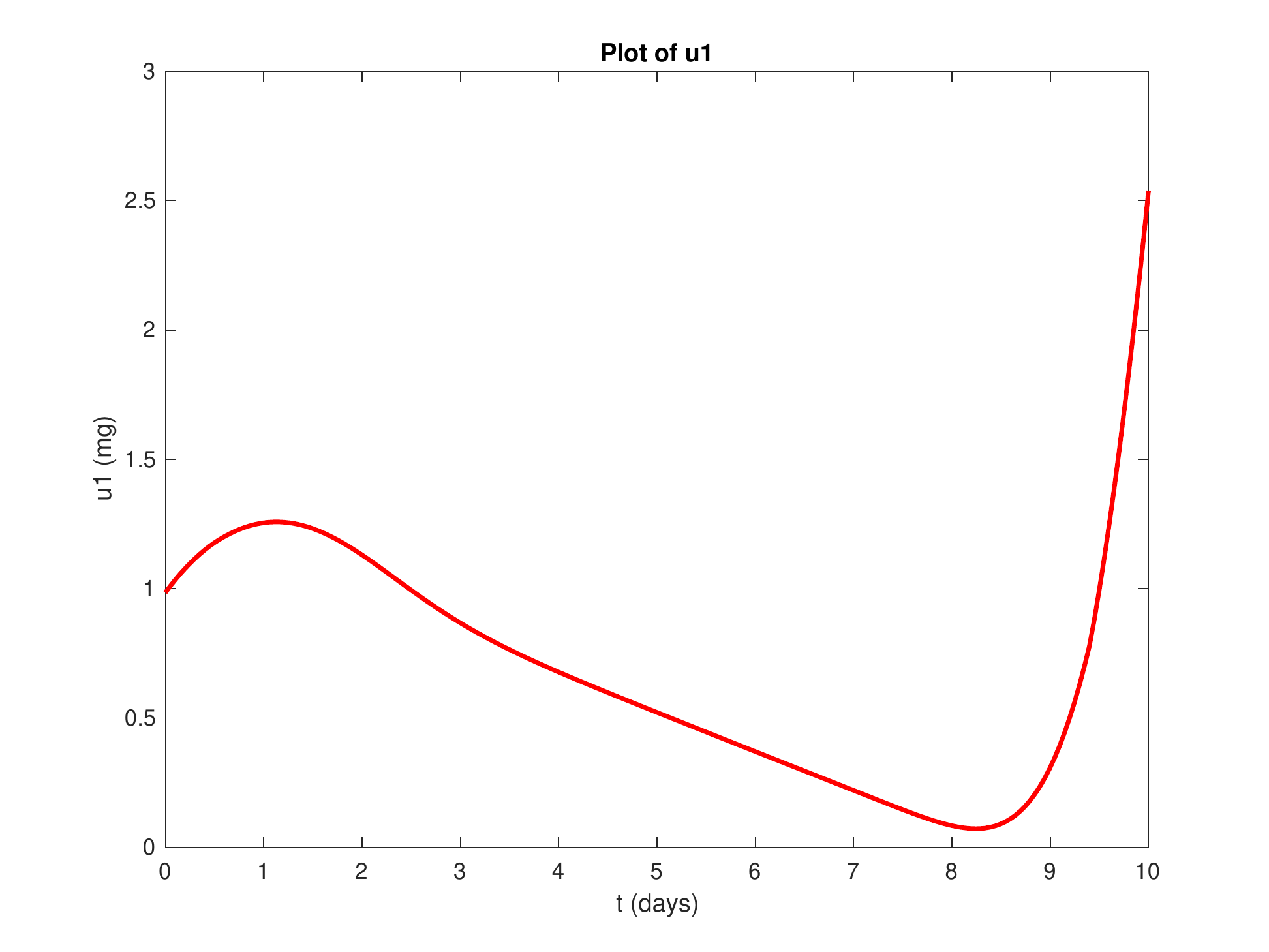}}
\subfloat[IL-2]{\includegraphics[width=0.25\textwidth, height=0.25\textwidth]{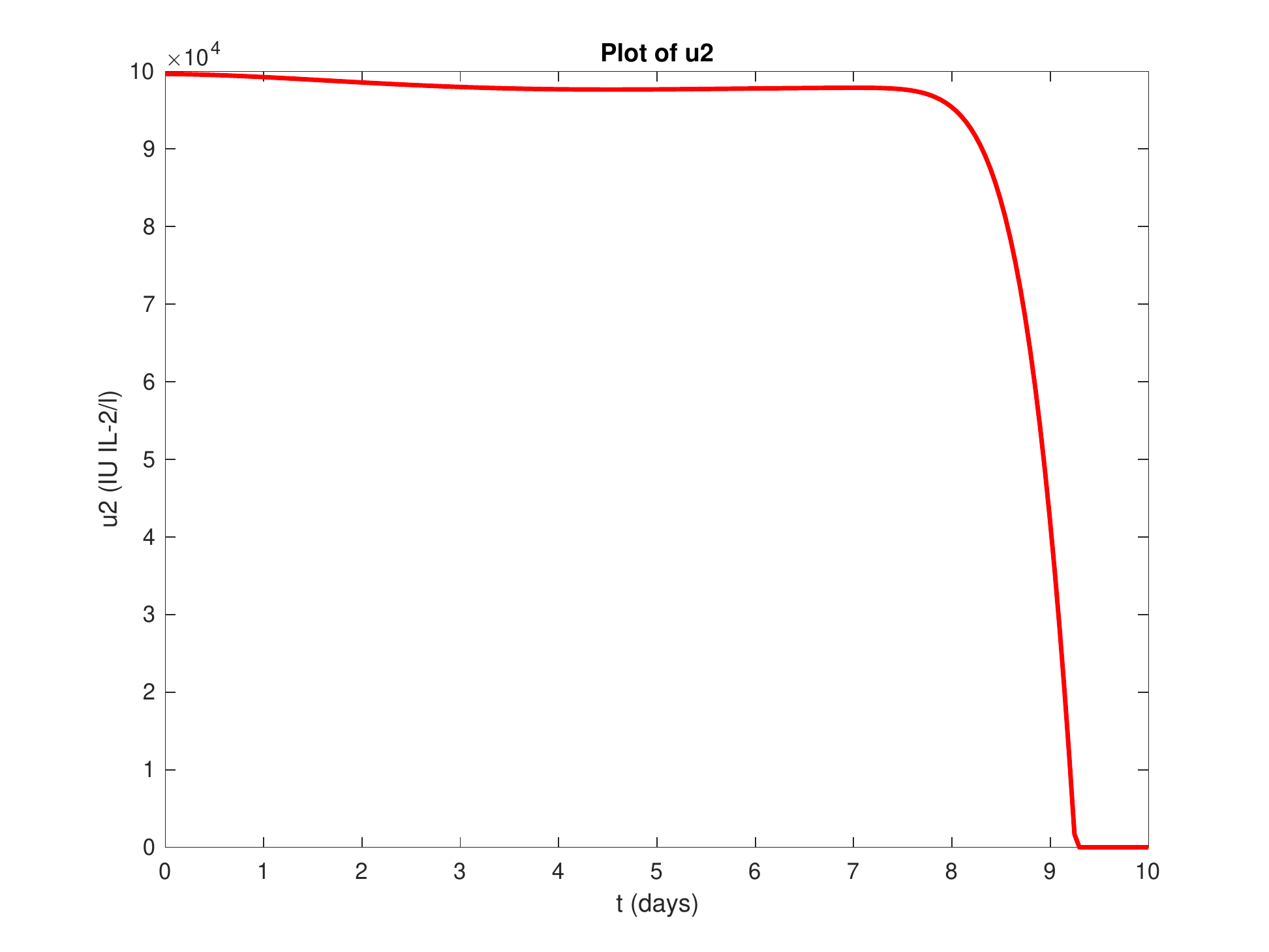}}\\

\caption{Test Case 2: Plots of the mean tumor profiles and drug dosages over the 10 day period}
    \label{fig:test_case2}
  \end{figure}
Figure \ref{fig:test_case2} shows the plots of the tumor profiles without and with treatment, and the optimal daily dosages of Doxorubicin and IL-2. We observe a similar behavior as Test Case 1. In addition, we note that the dosages are smaller compared to the patient in Test Case 1, since the immune system is moderately strong. This shows the effectiveness of our framework in obtaining optimal dosages in colon cancer.

\subsection{Conclusion}

 In this paper, we presented a new stochastic framework to determine optimal combination therapies in colon cancer-induced immune response. We considered the tumor and immune system response dynamics as proposed in \cite{pillis2014} and extend it to a stochastic process to account for incorporating \ind{randomness} in the dynamics. We characterized the state of the stochastic process using the PDF, whose evolution is governed by the FP equation. We then solved an optimal control problem with open loop controls, to obtain the optimal combination \ind{dosages} involving chemotherapy and immunotherapy. Numerical results demonstrate the feasibility of our proposed framework to obtain small dosages of the combination drugs leading to lower \ind{toxicity} whilst preserving the effectiveness to eliminate the tumor.

\subsection*{Acknowledgments} 

The authors were supported by the National Institutes of Health (Grant Number: R21CA242933). S. Roy was also partially supported by the Interdisciplinary Research Program, University of Texas at Arlington,
Grant number: 2021-772. The content is solely the responsibility of the authors and does not necessarily represent the official views of the National Institutes of Health.

\bibliographystyle{abbrv}
\bibliography{Bibliography}

\newpage



\end{document}